\documentclass[reqno, 12pt]{amsart}%
\usepackage{amssymb}
\usepackage[nesting]{hyperref}
\usepackage[pdftex]{graphicx}
\usepackage{listings}
\usepackage{multirow}
\usepackage{placeins}
\usepackage{color}
\usepackage{subfigure}
\usepackage{lscape}
\usepackage{dsfont}
\usepackage{amsmath}
\usepackage{amsfonts}
\usepackage{tikz}
\usepackage{verbatim}
\usepackage{accents} 
\usepackage{todonotes}
\usepackage{mathtools}
\usepackage{bbm}

\newcommand{\BlackBoxes}{\global\overfullrule5pt}

\BlackBoxes

\providecommand{\U}[1]{\protect\rule{.1in}{.1in}}

\newcommand{\R}{\mathbb{R}} 

\newcommand{\PP}{\mathbb{P}}
\newcommand{\EE}{\mathbb{E}}
\newcommand{\E}{\mathbb{E}}

\newtheorem{theorem}{Theorem}
\newtheorem{corollary}[theorem]{Corollary}
\newtheorem{lemma}[theorem]{Lemma}
\newtheorem{proposition}[theorem]{Proposition}
\theoremstyle{definition}

\newtheorem{remark}[theorem]{Remark}
\newtheorem{definition}[theorem]{Definition}

\numberwithin{equation}{section} \numberwithin{theorem}{section}
\def\0{\kern0pt\-\nobreak\hskip0pt\relax}

\makeatletter
\AtBeginDocument{ \def\@serieslogo{ \vbox to\headheight{ \parindent\z@ \fontsize{6}{7\p@}\selectfont
\vss}}}

\def\makeoverbar#1#2#3#4#5#6#7{ \setbox0=\hbox{$\m@th#2\mkern#5mu{{}#3{}}\mkern#6mu$} 
\setbox1=\null \dimen@=#4\fontdimen8#13 \dimen@=3.5\dimen@
\advance\dimen@ by \ht0 \dimen@=-#7\dimen@ \advance\dimen@ by \wd0
\ht1=\ht0 \dp1=\dp0 \wd1=\dimen@
\dimen@=\fontdimen8#13 \fontdimen8#13=#4\fontdimen8#13
\rlap{\hbox to \wd0{$\m@th\hss#2{\overline{\box1}}\mkern#5mu$}}
\fontdimen8#13=\dimen@}
\def\mylabel#1#2{{\def\@currentlabel{#2}\label{#1}}}
\makeatother
\begin{document}
\title[  ]{Time-consistency in the mean-variance problem: A new perspective}
\author[N. \smash{B\"auerle}]{Nicole B\"auerle${}^*$}
\address[N. B\"auerle]{Department of Mathematics,
Karlsruhe Institute of Technology (KIT),  Karlsruhe, Germany}

\email{\href{mailto:nicole.baeuerle@kit.edu}{nicole.baeuerle@kit.edu}}

\author[A. \smash{Ja\'skiewicz}]{Anna  Ja\'skiewicz${}^*$}
\address[A. Ja\'skiewicz]{Faculty of Pure and Applied Mathematics, 
Wroc{\l}aw University of Science and Technology,  Wroc{\l}aw, Poland}

\email{\href{mailto:anna.jaskiewicz@pwr.edu.pl} {anna.jaskiewicz@pwr.edu.pl}}


\begin{abstract}
We investigate discrete-time mean-variance portfolio selection problems \linebreak viewed as a Markov decision process. We transform the problems
into a new model with deterministic transition function for which the Bellman optimality equation holds. 
In this way, we can solve the problem recursively and obtain a time-consistent solution, that is an optimal solution that meets the Bellman optimality principle. 
We apply our technique for solving explicitly a more general framework. 

\end{abstract}
\maketitle


\makeatletter \providecommand\@dotsep{5} \makeatother



\vspace{0.5cm}
\begin{minipage}{14cm}
{\small
\begin{description}
\item[\rm \textsc{ Key words}]
{\small  Bellman optimality principle, time-consistency, mean-variance portfolio selection, linear-quadratic optimal control}

\end{description}
}
\end{minipage}

\section{Introduction}

The Bellman optimality principle is one of the key results in optimal control theory. It asserts that the optimization of
dynamic processes with finite time horizon is equivalent to the solution of a recursive equation also known  as 
a  dynamic program, dynamic programming problem or Bellman optimality equation. The part of operations research
concerned with Markov decision processes (MDPs) is sometimes called dynamic programming. 

From the dynamic programming principle, it follows that an optimal policy established for the initial points (time and state)
is also optimal along the optimal trajectory.  This property is referred to as the time-consistency of the optimal policy, see Subsec.\ 1.2
in \cite{bjork2021}.  However, in reality, the Bellman principle often fails. It occurs when a dynamic optimization problem
is not time additively separable, or simply, 
non-separable.  Mathematically, it means that  the optimal value cannot be separated into the sum
of the present payoff and the optimal value term derived from tomorrow. The mean-variance optimization criterion belongs to this class
of non-separable models. The nonlinear term of the (conditional) variance in the cost functional causes the time-inconsistency, 
since it spoils the iterated-expectations property.  For Markov decision models, this problem was analyzed in different settings,   for instance, 
in \cite{filar1989variance, sobel1994mean} and \cite{mannor2011mean} for  discrete-time processes  and in \cite{guo2012mean}  
for  continuous-time processes.  In \cite{filar1989variance},  the method,  based on a  penalty for the variability in the stream of rewards, allows to
 obtain a solution  via linear programming, whereas  in \cite{sobel1994mean}  
Pareto optima in the sense of high mean and low variance of the stationary probability distribution  are considered. The objective of 
\cite{mannor2011mean}, on the other hand, is to compare different policy classes and discusses computational complexity.

A prominent problem in the financial economics for this optimization criterion is the mean-variance portfolio selection model.
There are two main approaches that provide a solution. In the first one, the investor fixes the initial point (time and state) and finds 
the optimal policy that minimizes his objective cost functional.  In other words,
the agent pre-commits  to follow the derived strategy, although at future dates it is not longer optimal.
The second solution is based on a game theoretic approach. 
Namely,  the investor plays an intrapersonal game, whose outcome is a subgame 
perfect equilibrium. To apply this technique, the objective function is decomposed into the investor's expected future objective plus
an additional term that controls his deviations. 
This method leads to a certain recursive formula, which produces a suboptimal solution. Therefore, 
the  obtained strategy is  sometimes  called  dynamically optimal \cite{basak2010dynamic}  or
time-consistent dynamically optimal \cite{vigna} or time-consistent 
\cite{czichowsky2013time}. 
However, in this paper we follow the definition of the time-consistency given in \cite{bjork2021} (see also Definition \ref{tc}).

The literature contains a number of examples of the outlined approaches. 
For obvious reasons, we describe here only some positions.  The paper of Li and Ng
\cite{li2000optimal}  is the first work on a dynamic mean-variance portfolio model in discrete-time. The optimal policy is obtained
by an embedding method that transforms the original problem into a standard linear-quadratic (LQ)  model. 
This leads to time-inconsistency of an investor's strategy.  Alternatively, as noted by Yoshida \cite{yoshida2018solution},
the multi-period mean-variance portfolio selection problem can be  solved explicitly 
under the so-called deterministic mean-variance trade-off condition. 
The latter approach  that  generates a subgame perfect equilibrium 
was applied in \cite{bjork2021}  for discrete-time  models and in \cite{basak2010dynamic, bjork2014mean} for continuous-time models.
All these three papers  handle a Markovian setting, whereas the non-Markovian framework was studied in \cite{czichowsky2013time}.
Finally, there is also a possibility  to optimize the conditional mean-variance criterion
 myopically in each
step separately over the payoffs in the next period, see 
\cite{campbell2002strategic}.

Apart from the aforementioned methodologies, Cui et al. \cite{cui2014unified}  
suggests mean-field techniques to deal with time-inconsistency
in this model. The idea  is to
include the expected values of the state process into the underlying dynamic system and the objective functional. 
In such a way, the problem becomes separable but it is constrained. 
Their approach was generalized in
\cite{barbieri2020mean} for discrete-time linear systems with multiplicative noise. 
As an application,  the multi-period  optimal control policy for a general mean-variance problem was derived. 
An alternative approach,  based on mean-field theory and establishing 
the pre-commitment optimal solution in the continuous time mean-variance portfolio selection problem
was provided in \cite{bensoussan2014mean}. In particular, it was proved that 
 the optimal policy satisfies a Hamilton-Jacobi-Bellman equation that also involves the probability 
 distribution of the corresponding optimal wealth process.  

In this paper, we start our analysis  from  the discrete-time mean-variance portfolio selection problem modelled by an MDP. 
Although the problem is 
time-inconsistent, we show that this is not true if more information is included. More precisely, our idea relies on  a
replacement of the original state space (the levels of the wealth) 
by the set of probability measures on it. This trick, in turn, allows us to transform the original MDP
into a purely deterministic model, for which we are able to write a standard Bellman equation. 
We refer to this new model as a {\it population version } MDP. Thus, we provide for the first time a  backward induction algorithm that  generates the time-consistent policy. Furthermore, we show that if the initial distribution 
is the Dirac delta concentrated at some point $x_0$,
then the investor's time-consistent strategy (i.e., the strategy that is optimal and satisfies the Bellman optimality principle)
from the {\it population version} MDP coincides with  the pre-commitment optimal strategy
given in \cite{li2000optimal} or \cite{bauerle2011markov}. 
We also discuss and compare the optimal policy with a suboptimal solution based on 
game theoretic tools. 

Next we apply our approach to more general time-inconsistent problems outlined in Subsec.\ 1.6 in \cite{bjork2021}. We deal with the 
case, where the cost functions in the objective functional  are independent of the initial points (time and state). Within
a {\it population version} MDP framework,
the Bellman optimality principle works and a  time-consistent policy exists. 
As an application, we solve an LQ regulator problem and compare our solution with the equilibrium strategy from \cite{bjork2021}. 
It is worthy to mention that LQ control models were also studied in the mean-field setting, 
see \cite{elliott2013discrete, ni2014indefinite, ni2016time}
or \cite{pham2016discrete}. However, as noted in \cite{cui2014unified},   
this direction requires new analytical tools and solution techniques.
Therefore, within this view our method demonstrates a  direct and simple way in dealing with time-inconsistent problems 
that improves the quality of the solutions. 

The paper is organized as follows. In Sec. \ref{sec:2} we present the multi-period portfolio selection problem and describe three approaches. Sec. \ref{sec:general} generalizes our techniques to
solve an MDP in the finite time-horizon, in which the cost functional contains a non-linear term. Our findings are applied to an LQ model in Sec. \ref{sec:application}. A relationship to the mean-field framework and possible extensions are discussed in Sec. 
\ref{sec:meanfield}. Finally, we conclude in Sec. \ref{sec:conclusion}. The last part, Sec. \ref{sec:appendix}, is devoted to solving one-step models from Sections \ref{sec:2} and \ref{sec:application}.

\section{Mean-variance problem} \label{sec:2}

We start with the definition of a {\it time-consistent} solution taken from \cite{bjork2021}.

\begin{definition} \label{tc}
Assume that we have an $N$-finite time horizon optimization problem $\mathcal{P}(0,x)$ starting at time point 
$0$ with an initial state $x.$  Suppose  that the optimal control is ${\bf u}=(u^{0,x}_0,u^{0,x}_1,\ldots, u^{0,x}_N).$ 
The optimal control is 
{\it time-consistent}, if it satisfies the Bellman optimality principle, i.e., it is also optimal on the time interval $\{n,\ldots, N\}$
for any $n\in\{1,\ldots,N-1\}.$
\end{definition}

According to this definition, a {\it time-consistent} control must be optimal for the problem $\mathcal{P}(0,x)$ and, roughly
speaking, ${\bf u}$ must be independent of the initial state. More formally, this means that
the optimal control for $\mathcal{P}(n,y)$
is $(u^{n,y}_n,\ldots,u^{n,y}_N),$ where $ u^{n,y}_k=u^{0,x}_k$ for all time points $n\le k\le N$ and states $x,y.$ 

The mean-variance problem is often cited as a problem, where the optimal policy is {\it not  time-consistent.}
 The reason for the time-inconsistency in this model is the conditional variance term. Although the tower property holds
 for the conditional expectation, it does not hold for the  conditional variance term. The reader is referred to 
 \cite{basak2010dynamic} or 
 \cite{czichowsky2013time} for the specific formula for this expression and further discussion.  
 This fact, however, implies that  a  Bellman equation is not
the standard Bellman equation anymore, because the value functions appearing
 on the left-hand side and right-hand side  are not the same. In other words, the formulae for these functions are different. 
 Hence, although we may use
backward induction for deriving a solution,  this solution need not be optimal, see Subsec. \ref{sec:sol2}.
The other approach is to find an optimal policy 
for the whole problem and not care at being optimal for any time interval $\{n,\ldots,N\},$ see Subsec. \ref{sec:sol1}. 
It  is known as a {\it pre-commitment}  policy and it is not {\it time-consistent}. 
However, we show that this optimal solution can be seen as {\it time-consistent} and this is  a question of information.
More precisely, it becomes {\it time-consistent,}  when we extend the state space. 
We will explain this in Subsec. \ref{sec:sol3}. First we present the model and recalling the established approaches. 

\subsection{The model}
Let $Y$ be a Borel set. We always equip it with its Borel $\sigma$-algebra $\mathcal{B}(Y).$ 
By $\Pr(Y)$ we denote the set of probability measures on $(Y,\mathcal{B}(Y)).$ 

Let us consider the following specific model which is taken from Sec.\ 4.6 in
\cite{bauerle2011markov}. We have a financial market with one riskless bond. 
We assume that the riskless bond has interest rate $i_{n+1}$ during the period  $[n,n+1)$, i.e.,
$$S_{n+1}^0=(1+i_{n+1})S^0_n\quad\mbox{for time points}\quad n=0,\ldots,N-1,\quad\mbox{where}\quad  S_0^0\equiv 1.$$  
In addition, we have $d$ risky assets with price processes
$$ S_{n+1}^k  = S_n^k \tilde R^k_{n+1},\quad k=1,\ldots,d, \quad n=0,\ldots,N-1$$ and $S_0^k\equiv 1$ for all $k,$ 
where $\tilde R^k_{n+1}$ are random variables on a joint probability space $(\Omega,\mathcal{ F},\PP)$, which are assumed to be almost sure positive. 
In what follows, we use the notation $S_n^\top =(S_n^1,\ldots,S_n^d)$ and 
$\tilde R_n^\top =(\tilde R_n^1,\ldots,\tilde R_n^d)$. 
We assume that the random vectors $\tilde R_1,\ldots \tilde R_N$ are independent. Further, 
let $R_n^k :=\frac{ \tilde R_n^k}{1+i_n}-1$ and $R_n^\top = (R_n^1,\ldots,R_n^d)$ be the relative risk process. 
We further assume that $\E\|R_n\|<\infty,$ $\E R_n \neq 0$ and that the covariance matrix
 $\Sigma_n:=\big(Cov(R_n^j,R_n^k)\big)$ of these random vectors $R_n$ is positive definite for all $n$.  
We can invest into this financial market. Let $A_n^\top =(A^1_n,\ldots,A_n^d)\in\R^d.$ 
Here, $A_n^k$ is the amount of money 
invested into asset $k$ during time interval $[n,n+1)$. Thus, if 
$X_n$ denotes the wealth of the agent at time $n,$ then
$X_n-A_n \cdot e,$ $e=(1,\ldots,1)^\top \in\R^d$ 
is the amount invested in the riskless bond, where $\cdot$ is the usual scalar product.  
Suppose that  the initial wealth  of the agent is $x_0\in \R.$  
Then, the evolution of the wealth is given by the following difference equation
\begin{equation}\label{eq:wealth}
    X_{n+1} = (1+i_{n+1})(X_n + A_n \cdot R_{n+1}).
\end{equation} 

The mean-variance model can be viewed as 
a non-stationary Markov decision process (MDP) with the following items:
\begin{itemize}
\item $(E,\mathcal{B}(E))$  is the state space with  $E:= \R$; here $x\in E$ denotes the wealth,
\item $(A,\mathcal{B}(A))$
 is the action space with $A :=\R^d;$ here $a\in A$ is the amount of money invested in the risky assets,
\item $\mathcal{R}:= [-1,\infty)^d$ is the space of values for the relative risk, i.e.,  $r\in\mathcal{R}$ denotes the relative risk,
\item $T_n(x,a,r)=(1+i_{n+1})(x+a\cdot r)$ is the transition function at time $n$, see \eqref{eq:wealth}, 
where $x\in E,$ $a\in A$ and $r\in \mathcal{R};$ the transition probability is given by
$$ Q_n(B|x,a)= \PP\big(T_n(x,a,R_{n+1})\in B\big), \quad\mbox{where } B\in\mathcal{B}(E).$$
\end{itemize}

In what follows, let $\Pi$ be the set of all deterministic policies $\pi=(\pi_0,\ldots,\pi_{N-1}).$
In other words, each $\pi_n,$ $n=0,\ldots, N-1,$ is a measurable mapping defined 
on the history space of the process up to the $n$-th state
with values in $A$.  Moreover, let
 $$ F := \{ \phi: E \to A \ | \ \phi \mbox{ is measurable}\}.$$ 
 Hence, $F$ is the set of decision rules that take into account only the
current wealth of the agent.  The sequence  $(\phi_0,\ldots,\phi_{N-1}),$ where $\phi_k\in F,$ $k=0,\ldots,N-1,$
determines a Markovian investment strategy to the agent. 

Let $(\widetilde{ \Omega}, \widetilde{\mathcal{F}})$   be a measurable space with 
	$\widetilde{\mathcal{F}}=\mathcal{B}(E)\otimes\cdots\otimes\mathcal{B}(E)$ ($N+1$ times).
 Then, for a fixed policy
$\pi=(\pi_{0},\ldots,\pi_{N-1})\in \Pi$  and an  initial
state $x_{0} \in E$, a probability measure $\PP^{\pi}_{x_0}$ 
is uniquely defined on $(\widetilde{ \Omega}, \widetilde{\mathcal{F}})$ 
according to the Ionescu-Tulcea theorem (see Appendix B in \cite{bauerle2011markov}).
By $\E^{\pi}_{x_0}$ we denote the expectation operator with respect to the measure
$\PP^{\pi}_{x_0}.$ 
 
The mean-variance problem is  given as a Lagrange problem by
$$ P(\lambda) \left\{ \begin{array}{l}
Var_{x_0}^\pi [X_N] -2\lambda \E_{x_0}^\pi [X_N] \to \inf\\
\pi\in \Pi
\end{array}\right.$$
where $\pi=(\pi_0,\ldots,\pi_{N-1})\in\Pi$. 
Though not being separable, the problem can be solved 
with techniques from the theory of MDPs. 
Before we state and compare different solutions, let us define 
\begin{eqnarray} \label{eq:ell}\nonumber
    C_n := \E[R_nR_n^\top],&\quad& \ell_n := (\E [R_n])^\top C_n^{-1} \E [R_n]\\
    d_N:=1,&\quad&d_n:=d_{n+1}(1-\ell_{n+1}),\quad n=0,\ldots, N-1. 
\end{eqnarray}
The matrix $C_n$ is positive definite and thus regular. It can also be shown that $\ell_n, d_n\in(0,1)$ 
(see Lemma 4.6.4 in \cite{bauerle2011markov}).

\subsection{Solution 1: Via auxiliary LQ-problems}\label{sec:sol1}
Problem $P(\lambda)$ can be solved with the help of an auxiliary problem 
$$ QP(b) \left\{ \begin{array}{l}
 \E_{x_0}^\pi [(X_N-b)^2] \to \inf\\
\pi\in F^N
\end{array}\right.$$
which is a standard LQ-problem. In $QP(b)$ it is clear that we can replace $\Pi$ by $F^N,$ 
since the optimal value will be attained under a Markovian investment strategy and  $QP(b)$ 
can be solved by backward induction. 
The optimal policy for $QP(b)$ is for $n=0,1,\ldots,N-1$ and $x\in\R$ given by 
\begin{equation}\label{eq:optimalpolicyQb}
 \phi_n^b(x)= \Big(\frac{bS^0_n}{S^0_N}-x\Big) C_{n+1}^{-1} \E R_{n+1}.
\end{equation} 
This was shown in Theorem 4.6.5(b) in \cite{bauerle2011markov}.
Further, we know that if $\pi^b=(\phi^b_0,\ldots,\phi^b_{N-1})$ is optimal for $P(\lambda)$, 
then $\pi^b$ will also be optimal for
$ QP(b) $ with $b=\E_{x_0}^{\pi^*} [X_N]+\lambda$ (see Lemma 4.6.3 in \cite{bauerle2011markov}).
Under $\pi^b$ the expected wealth is equal to
\begin{align}\nonumber
\E_{x_0}^{\pi^b} [X_N] &= \E_{x_0}^{\pi^b} [(1+i_N)(X_{N-1}+\phi_{N-1}^b(X_{N-1})\cdot R_{N})]\\ \nonumber
&= \E_{x_0}^{\pi^b} \Big[(1+i_N)\Big(X_{N-1}(1-R_N^\top C^{-1}_N\E R_N)+\frac{bS^0_{N-1}}{S^0_N}
R_N^\top C^{-1}_N\E R_N\Big)\Big]\\ \nonumber
&=(1+i_N)\Big( \E_{x_0}^{\pi^b} [X_{N-1}](1-\ell_N)+\frac{b S^0_{N-1}}{S_N^0}\ell_N\Big)\\\nonumber
&=
(1+i_N)d_{N-1}
\E_{x_0}^{\pi^b} [X_{N-1}]+b \ell_N\\ \nonumber
&=
(1+i_N)d_{N-1}
(1+i_{N-1}) \Big(\E_{x_0}^{\pi^b} [X_{N-2}](1-\ell_{N-1})+\frac{b S^0_{N-2}}{S_N^0}\ell_{N-1}\Big)
+b \ell_N\\ \nonumber
&=
(1+i_N)(1+i_{N-1})d_{N-2}\E_{x_0}^{\pi^b} [X_{N-2}] +b\ell_{N-1}d_{N-1}+b\ell_N=\ldots\\\nonumber
&= \label{eq:EX}
\prod_{k=1}^N(1+i_k)d_0x_0 + b\sum_{k=1}^N\ell_kd_k
= 
S^0_Nd_0x_0 + b\sum_{k=1}^N\ell_kd_k\\ &=S^0_Nd_0x_0 + b(1-d_0).
\end{align}
The last equality is due to the definition of $d_n$. Since $\ell_k d_k = d_k-d_{k-1}$ we obtain
\begin{eqnarray*}
\sum_{k=1}^N\ell_kd_k&=& \sum_{k=1}^N (d_k-d_{k-1}) = (d_N-d_0)=1-d_0.
 \end{eqnarray*}
Thus, when we set $b=\E_{x_0}^{\pi^b} [X_N]+\lambda,$ we obtain 
\begin{equation}\label{valb}
 b=  S_N^0x_0  +\frac{ \lambda}{d_0}.
 \end{equation}
Plugging \eqref{valb} into \eqref{eq:optimalpolicyQb}, we have partly proved the following:
\begin{proposition} \label{sol1}
An optimal policy $\pi^o=(\phi^o_1,\ldots,\phi^o_{N-1})$ for $P(\lambda)$ is given by
 $$\phi_n^o(x) =\Big(\Big(x_0  +\frac{ \lambda}{S_N^0d_0}\Big)S^0_n-x\Big) C_{n+1}^{-1} \E R_{n+1}, \quad x\in E,$$
 for  $n=0,1,\ldots,N-1,$
and the value for $P(\lambda)$ is
$$V^o(x_0):= \lambda^2\Big(1-\frac 1{d_0}\Big)-2 \lambda x_0 S_N^0.$$
\end{proposition}

\begin{proof} 
It remains to show the optimal value for $P(\lambda).$  From Theorem 4.5.6(c) in \cite{bauerle2011markov}, 
we get under $\pi^b$ from \eqref{eq:optimalpolicyQb} (see also the previous computation in \eqref{eq:EX}):
\begin{eqnarray*}
\E_{x_0}^{\pi^b} [X_N]&=&x_0S_N^0d_0+b(1-d_0),\\
\E_{x_0}^{\pi^b} [X_N^2]&=& (x_0S_N^0)^2d_0+b^2(1-d_0).
\end{eqnarray*}
Plugging \eqref{valb} into above expressions, we easily get the conclusion. 
\end{proof}

This policy is considered to be time-inconsistent, because the initial wealth $x_0$  enters the formula.
Such a  type of policy (depending on the current and initial states)  is in MDP theory called  a semi-Markov policy.
In the economic and financial literature, it  is sometimes known as a {\it pre-commitment} strategy. 
In other words, the agent cares only at being time consistent at the initial point $x_0$
and does not care of being time-consistent at future time points $n=1,\ldots,N-1.$ More precisely, if the wealth at time $n=1$ is $x_1\neq x_0$, 
then the previously computed restriction $(\phi^o_1,\ldots,\phi^o_{N-1})$  is not optimal for the remaining time horizon.

\begin{remark}
An alternative approach is to write $P(\lambda)$ as 
$$ P(\lambda) \left\{ \begin{array}{l}
\min_{b\in\R} \E_{x_0}^\pi\Big[(X_N-b)^2\Big] -2\lambda \E_{x_0}^\pi [X_N] \to \inf\\
\pi\in \Pi
\end{array}\right.$$
using the fact that $Var(X)= \min_{b\in\R} \E[(X-b)^2]$. Since we can interchange the infima, we can first solve 
$$ \inf_{\pi\in\Pi} \E_{x_0}^\pi\Big[(X_N-b)^2\Big] -2\lambda \E_{x_0}^\pi [X_N] $$ 
and afterwards minimize w.r.t.\ $b$. This yields exactly the same solution  and we do not repeat the computation here.
\end{remark}

\subsection{Solution 2: Equilibrium strategies} \label{sec:sol2}
Another way to deal with this problem 
 is to look for equilibrium strategies, 
see Sec. 5 in  \cite{bjork2021}. These strategies are computed by backward induction with the aid of a
so-called system of Bellman equations. In this way, we obtain a policy which 
is sometimes called dynamically optimal (see \cite{vigna}) or a subgame perfect equilibrium,
but it need not be  {\it time-consistent} according to Definition \ref{tc}, since it is not optimal for the problem 
$P(\lambda)$. We show this fact below. 

Bj\"ork et al. \cite{bjork2021} conjecture first  that an equilibrium strategy is 
linear. Then, plugging this finding into their system of equations, they 
solve the problem inductively starting from the  time point $N-1.$  
Using their algorithm to our setting,  we obtain  the following:

\begin{proposition}
An equilibrium strategy  $\pi^{e}=(\varphi_0,\ldots,\varphi_{N-1})$ for $P(\lambda)$ is given by
 $$\varphi_n(x)\equiv  \lambda \frac{S_n^0}{S_N^0} (\E R_{n+1})^\top \Sigma_{n+1}^{-1},\quad  x\in E,$$ 
 for  $n=0,1,\ldots,N-1,$
and the value 
under this equilibrium strategy for $P(\lambda)$ is
$$V^{e}(x_0):= -2\lambda x_0S_N^0- \lambda^2 \sum_{n=1}^N (\E R_n)^\top \Sigma_n^{-1} 
\E R_n = -2\lambda x_0S_N^0- \lambda^2 \sum_{n=1}^N  \frac{\ell_n}{1-\ell_n}$$
with $\ell_n$ from \eqref{eq:ell}. 
\end{proposition}

This strategy does not depend on the wealth, but only on the time point $n$ of a decision. 
Note that the last equation in the formula for $V^{e}(x_0)$ follows from the fact that according to the Sherman-Morrison formula
\begin{align*}
 \Sigma_n^{-1} = C_n^{-1} +\frac{\Sigma_n^{-1} \E R_n \E R_n^\top \Sigma_n^{-1}}{1+\E R_n^\top \Sigma_n^{-1} \E R_n}.
\end{align*}
Thus,
\begin{align*}
 \E R_n^\top \Sigma_n^{-1} \E R_n & =\ell_n   
 +\frac{\big(\E R_n^\top \Sigma_n^{-1} \E R_n\big)^2}{1+\E R_n^\top \Sigma_n^{-1} \E R_n},
\end{align*}
which implies 
\begin{align*}
  \E R_n^\top \Sigma_n^{-1} \E R_n & = \frac{\ell_n}{1-\ell_n}.
\end{align*}
When comparing the values of Solutions 1 and 2 we get
\begin{eqnarray*}
&V^{o}(x_0)<V^{e}(x_0)\\
\Leftrightarrow&-2\lambda x_0S_N^0- \lambda^2\Big(\frac{1-d_0}{d_0}\Big)
<  -2\lambda x_0S_N^0- \lambda^2 \sum_{k=1}^N  \frac{\ell_k}{1-\ell_k}\\
\Leftrightarrow & \frac{1-d_0}{d_0} >  \sum_{k=1}^N  \frac{\ell_k}{1-\ell_k}\\
\Leftrightarrow & \frac{1-\prod_{k=1}^N (1-\ell_k)}{\prod_{k=1}^N (1-\ell_k)} >  \sum_{k=1}^N  \frac{\ell_k}{1-\ell_k}\\
\Leftrightarrow & \prod_{k=1}^N(h_k+1) -1>\sum_{k=1}^N h_k   
\end{eqnarray*}
for positive  $h_k=\ell_k/(1-\ell_k).$
The latter inequality is true by a version of the  Weierstrass inequality, see Sec. 7 in \cite{mitrinovic}.

The difference of the two values for the single-asset case and different time horizons is given in Figure \ref{fig:diff}. 
We present 
a stationary model with  two-point distribution of  $\tilde R$ given by $\PP(\tilde R = 1+4i)=\PP(\tilde R=1-i)=\frac12,$ 
where $i$ is an interest rate in each period, 
and $\lambda=1.$  In this case, $\ell = \frac{1}{26}$ 
and the difference does not depend on $i$. 
For $N=1$, the two-values coincide. However, the difference gets larger, the larger the time-horizon is. 
Thus, the equilibrium strategy  is suboptimal for $P(\lambda),$ although is found by backward induction.  
Since it is suboptimal, it is not time-consistent according to Definition \ref{tc}. 
\begin{figure}
\centering
\includegraphics[scale=0.4]{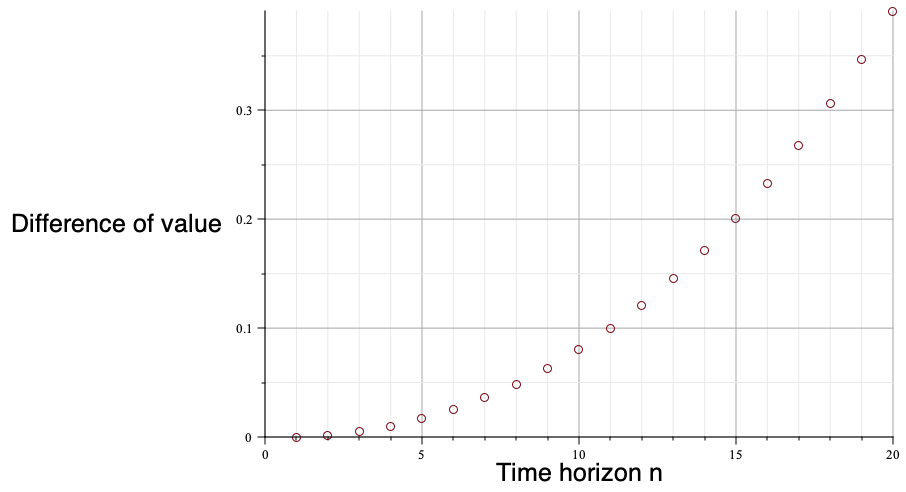}
 \caption{Difference of target values under the optimal policy 
and the equilibrium strategy for a stationary model with $\ell = 1/26$ for different time horizons.}
 \label{fig:diff}
\end{figure}

\subsection{Solution 3: Population version}  \label{sec:sol3}
In this part, we show that the solution obtained in Proposition \ref{sol1} can be seen as a time-consistent solution
provided that we enlarge the state space. More precisely, the idea relies on replacing the state space $E$ 
by the set of probability measures on $E.$ 
In this way, the problem becomes purely deterministic and more importantly, 
as can be seen below, the Bellman principle holds.
The advantage of our approach is that we can solve the problem directly as an 
MDP without making use of any auxiliary problem. 
Using backward induction we obtain a {\it time-consistent } policy for the investor in this {\it population version} model. 

In our new setting,  a {\it population version } MDP is defined by the following items:
\begin{itemize}
\item $\hat E := \Pr(E)$ is the state space, where $\mu\in \hat E$ denotes the distribution of  the wealth,
\item $ \hat A :=F$ is the action space, where $\phi\in \hat A$ is the decision rule to be applied,
\item $\hat T_n(\mu, \phi)=\mu'$ is 
the transition function, where 
$$\mu'(B)= \int \PP\big((1+i_{n+1})(x+\phi(x) \cdot  R_{n+1}) \in B\big)\mu(dx) = \int Q_n^\phi(B|x) \mu(dx) $$ 
for a Borel set $B\subset E,$ where $Q_n^\phi(B|x) :=Q_n(B|x,\phi(x)),$ $x\in E;$
alternatively,
 $\mu' = \mathcal{L}\big((1+i_{n+1})(X+\phi(X)\cdot R_{n+1})\big)$ with $X\sim \mu$
independent of $R_{n+1}$, 
\item $\hat c_N(\mu)= \int (x^2 -2\lambda x)\mu(dx) -\Big( \int x\mu(dx)\Big)^2 $ is the terminal cost.
\end{itemize}
The transition functions $\hat T_n$ are purely deterministic, which implies that this problem can 
be seen as a deterministic dynamic program. Without loss of generality, 
we restrict to the Markovian decision rules $\hat\phi_n\in\hat F$, where
$$ \hat F := \{ \hat\phi: \hat E \to \hat A \}.$$
We place no measurability requirements on these mappings.

The aim is now to solve
$$ \hat P(\lambda) \left\{ \begin{array}{l}
Var_{\mu_0}^\pi [X_N] -2\lambda \E_{\mu_0}^\pi [X_N] \to \inf\\
\pi\in \Pi
\end{array}\right.$$
This means that the variance and expectation 
are taken w.r.t.\ the product measure 
 $\PP_{\mu_0}^\pi$  defined by $\mu_0$ and 
 the transition kernels of $(X_n)$ generated by the policy $\pi=(\pi_0,\ldots,\pi_{N-1}).$ 
Clearly, the problem $\hat P(\lambda) $ becomes  non-trivial if $N\ge 2.$

Observe that the initial capital is drawn at random at the beginning, i.e., $X_0,$ is not deterministic, 
but has a distribution $\mu_0$. In case $\mu_0=\delta_{x_0},$ the problem reduces to the initial problem $P(\lambda).$ 
An alternative interpretation is that we consider a continuum of agents 
with wealth distributed according to $\mu_0.$

\begin{definition} \label{def1}
For a fixed policy $\hat \pi = (\hat\phi_0,\ldots,\hat\phi_{N-1})$ define
\begin{align}
\nonumber
 \phi_0 &:= \hat\phi_0(\mu_0),\quad 
 \mu_1 := \hat T_0(\mu_0,\phi_0),\\  \nonumber
   \phi_1 &:= \hat\phi_1(\mu_1),\quad
 \mu_2 := \hat T_1(\mu_1,\phi_1),\\  \nonumber
 &\vdots\\ \label{eq:sequence}
 \phi_{N-1} &:= \hat\phi_{N-1}(\mu_{N-1}),\quad \mu_{N} := \hat T_{N-1}(\mu_{N-1},\phi_{N-1}).
\end{align} 
The sequence of actions defines a policy $\pi=(\phi_0,\phi_1,\ldots,\phi_{N-1})$ in the original problem formulation. 
\end{definition}

From this definition we immediately have:

\begin{lemma} \label{lemp1}
Under policy $\hat\pi=(\hat\phi_0,\ldots ,\hat\phi_{N-1})$, the distribution of $X_k$ is equal to $\mu_k$,
where $\mu_k$ are defined in (\ref{eq:sequence}) for $k=0,1,\ldots, N.$
\end{lemma}

\begin{proof}
 The proof is by induction. For
 $k=0,$ the statement is obvious.  
 Suppose that $X_k\sim \mu_k$ for $k=1,\ldots ,n.$ Then, 
 $$\mu_{n+1}=\hat T_n(\mu_n,\hat\phi_n(\mu_n))=\hat T_n(\mu_n,\phi_n)=
 \int \PP\big((1+i_{n+1})(x+\phi_n(x) \cdot R_{n+1}) \in 
\bullet\big)\mu_n(dx).$$
 By construction, the distribution of $X_{n+1}$ is given by $\mu_{n+1}.$ This implies the result.
\end{proof}

Next for a fixed policy $\hat \pi = (\hat\phi_0,\ldots,\hat\phi_{N-1})$  define 
\begin{eqnarray*}
 J_{N,\hat\pi}(\mu) &:=& \hat c_N(\mu),\\
 \mbox{and}\quad J_{n,\hat\pi}(\mu) &:=&  J_{n+1,\hat\pi}(\hat T_{n}(\mu,\hat\phi_{n}(\mu))) \quad\mbox{for  }
  n=1,\ldots,N.
\end{eqnarray*}
Then, again making use of induction we obtain:

\begin{lemma} \label{lemp2}
Under policy $\hat \pi = (\hat\phi_0,\ldots,\hat\phi_{N-1})$, 
$$  J_{0,\hat\pi}(\mu_0) = Var_{\mu_0}^{\pi} [X_N] -2\lambda \E_{\mu_0}^{\pi} [X_N],$$
where $\pi$ on the right-hand side  is defined as in \eqref{eq:sequence}.
\end{lemma}

\begin{proof}
Indeed, note that
$$
J_{0,\hat\pi}(\mu_0) = J_{1,\hat\pi}(\hat T_0(\mu_0,\hat\phi_{0}(\mu_0)))=J_{1,\hat\pi}(\mu_1)=\ldots=
J_{N,\hat\pi}(\mu_N)\\
   =\hat c_N(\mu_N).
$$
The claim now  follows from the definition of $\hat c_N$ and Lemma \ref{lemp1}.
\end{proof}

Finally, we define for $n=0,\ldots,N-1$:
\begin{align}\nonumber
 J_{N}(\mu) &:= \hat c_N(\mu),\\ \label{eq:recursion}
 J_{n}(\mu) &:= \inf_{\hat\phi\in \hat F} J_{n+1}(\hat T_{n}(\mu,\hat\phi(\mu))).
\end{align}
If the  minimum in \eqref{eq:recursion} is
attained, we denote it by $\hat \phi^*_{n}(\mu).$ Using backward induction, we get the following result.

\begin{theorem}\label{theo:main_mv} 
\begin{itemize}
\item[(i)]  If the initial distribution is $\mu_0,$ then the value of the problem  $\hat P(\lambda)$ is $J_0(\mu_0)$    and 
$$ J_{0}(\mu_0) =  (S^0_N)^2 d_0 \Big( \int x^2 \mu_0(dx) - \big(\int x\mu_0(dx)\big)^2\Big) -
2\lambda S_N^0 \int x\mu_0(dx)-\lambda^2\Big(\frac{1}{d_0}-1\Big).$$
\item[(ii)] The optimal policy in the population version model  
is given by $\hat \pi^*=(\hat\phi^*_0,\ldots ,\hat\phi^*_{N-1})$, where
$$ \hat \phi^*_{n}(\nu)(x)= 
\left(\frac{\lambda  S^0_n}{d_n S^0_N }+\int x'\nu(dx')  -x\right)C_{n+1}^{-1}\E R_{n+1},
$$ for 
$ n=0,\ldots,N-1,$ $\nu\in \hat E$ and $x\in E.$
\end{itemize}
\end{theorem}

\begin{proof}   
We wish to find an optimal decision rule $\hat\phi^*_{N-1}$ at time point 
$N-1.$ Let  $X_{N-1}\sim \nu.$ Then,
\begin{align}\label{eq:onestep} \nonumber
J_{N-1}(\nu)=& \inf_{\hat\phi\in \hat F} J_N(\hat T_{N-1}(\nu,\hat\phi(\nu)))=
\inf_{\phi\in F} \hat c_N(\hat T_{N-1}(\nu,\phi))\\=& 
\inf_{\phi\in F}\left( \E_{N,\hat T_{N-1}(\nu,\phi)} \Big[(X_{N})^2-2\lambda X_N\Big]- 
\Big(\E_{N,\hat T_{N-1}(\nu,\phi)}[X_N]\Big)^2\right)
\\  \nonumber
=&(1+i_N)^2  \inf_{\phi\in F} \E_{N-1,\nu}\Big[(X_{N-1} +\phi(X_{N-1}) \cdot R_N)^2+ \\&-\frac{2 \lambda} {1+i_N}(X_{N-1}+\phi(X_{N-1})\cdot R_N) 
-\Big(\E_{N-1,\nu} \Big[X_{N-1}+\phi(X_{N-1})\cdot R_N\Big]\Big)^2 \Big].
 \nonumber
\end{align} 
We use here the notation $\E_{n,\nu}$ to indicate that we start at the time point $n$ and the random variable 
$X_n$ has distribution $\nu.$  In case $n=0,$ we simply write $\E_\nu.$
By Appendix \ref{sec:appendix1} 
the solution of the problem is  
$$ \hat \phi^*_{N-1}(\nu)(x)= \left(\frac{\lambda S_{N-1}^0}{d_{N-1}S_N^0}+\int x'\nu(dx')  -x\right)C_{N}^{-1}\E R_{N}.$$
This also implies that
\begin{eqnarray*}
 J_{N-1}(\nu)
 &=& d_{N-1}\Big(\frac{S_N^0}{S_{N-1}^0}\Big)^2 \Big(\E_{N-1,\nu}[X_{N-1}^2] -(\E_{N-1,\nu} [X_{N-1}])^2+\\
 && \hspace*{3cm}-\frac{2\lambda}{d_{N-1}S_N^0/S_{N-1}^0}
  \E_{N-1,\nu} [X_{N-1}] \Big) -\lambda^2
\frac{\ell_N}{1-\ell_N}.
\end{eqnarray*}
At time $N-2$ the expression which has to be minimized is the same as at time $N-1,$ 
when we replace $\lambda$ by $\frac{\lambda S^0_{N-1}}{d_{N-1}S^0_N}$ in \eqref{eq:onestep}. Indeed, observe that 
\begin{eqnarray*} 
J_{N-2}(\nu)&=&
\inf_{\phi\in F}  J_{N-1}(\hat T_{N-2}(\nu,\phi))\\
&=& d_{N-1}\Big(\frac{S_N^0}{S_{N-1}^0}\Big)^2
\inf_{\phi\in F}
\Big(\E_{N-1,\hat T_{N-2}(\nu,\phi)}[X_{N-1}^2] -(\E_{N-1,\hat T_{N-2}(\nu,\phi)} [X_{N-1}])^2+\\
 && \hspace*{1cm}-\frac{2\lambda}{d_{N-1}S_N^0/S_{N-1}^0}
  \E_{N-1,\hat T_{N-2}(\nu,\phi)} [X_{N-1}] \Big) -\lambda^2
\frac{\ell_N}{1-\ell_N}. 
\end{eqnarray*} 
We obtain
\begin{eqnarray*}
  \hat \phi^*_{N-2}(\nu)(x)
  &=& \left(\frac{\lambda S^0_{N-2}}{d_{N-2}S^0_N} +\int x'\nu(dx')  -x\right)C_{N-1}^{-1}\E R_{N-1}
  \end{eqnarray*}
and 
\begin{eqnarray*}
J_{N-2}(\nu) &=&d_{N-2}\Big(\frac{S_N^0}{S_{N-2}^0}\Big)^2 \Big(\E_{N-2,\nu}[X_{N-2}^2] -(\E_{N-2,\nu} [X_{N-2}])^2
\\
&&-\frac{2\lambda}{d_{N-2}S_N^0/S_{N-2}^0} \E_{N-2,\nu} [X_{N-2}] \Big)
-\lambda^2\left(\frac {\ell_N}{1-\ell_N}+\frac{\ell_{N-1}}{(1-\ell_{N-1})(1-\ell_N)}\right).
\end{eqnarray*}
Continuing this procedure we finally get
$$  \hat \phi^*_{0}(\nu)(x)= \left(\frac{\lambda}{d_0 S^0_N} +\int x'\nu(dx')  -x\right)C_{1}^{-1}\E R_{1}$$
and 
\begin{eqnarray*}
J_{0}(\nu) &=&  d_{0}\Big(\frac{S_N^0}{S_{0}^0}\Big)^2 \left(\E_{\nu}[X_{0}^2] -(\E_{\nu} [X_{0}])^2
-\frac{2\lambda}{d_{0}S_N^0} \E_{\nu} [X_{0}] \right)+\\&&-\lambda^2\left(\frac {\ell_N}{1-\ell_N}
+\frac{\ell_{N-1}}{(1-\ell_N)(1-\ell_{N-1})}+\ldots+\frac{\ell_{1}}{(1-\ell_N)\cdots(1-\ell_1)}\right)\\
&=&d_{0}(S_N^0)^2 Var_{\nu}[X_{0}] -2\lambda S_N^0 \E_{\nu} [X_{0}] -\lambda^2\left(\frac 1{d_0}-1\right).
\end{eqnarray*}
Note that the last equality follows from
\begin{eqnarray*}
    \sum_{k=1}^N \frac{\ell_k}{(1-\ell_N)\ldots (1-\ell_k)} &=&
   \sum_{k=1}^N \frac{\ell_k}{d_{k-1}} = \sum_{k=1}^N \frac{d_k-d_{k-1}}{d_{k-1}d_k}\\
    &=&  \sum_{k=1}^N \Big(\frac{1}{d_{k-1}}-\frac{1}{d_k}\Big)=\frac{1}{d_0}-\frac{1}{d_N}=\frac{1}{d_0}-1.
\end{eqnarray*}
Hence, we have defined the optimal policy  
$\hat \pi^*=(\hat\phi^*_0,\ldots,\hat\phi^*_{N-1}).$
Now, if $\mu_0$ is an initial distribution, then $J_{0}(\mu_0)$ is optimal value for $\hat P(\lambda).$ 
\end{proof}

The policy $\hat \pi^*$ obtained in Theorem \ref{theo:main_mv} is {\it time-consistent}  in the {\it population version} model 
due to Definition \ref{tc}.   This is because  the same functions $J_n$ and $J_{n+1}$
appear on the left-hand  and  right-hand side in  \eqref{eq:recursion}. 
This immediately implies that the Bellman optimality principle 
holds and backward induction in this case produces an optimal policy. 

Assume that $\mu_0\in\hat E$ is given. Then,
making use of $\hat \pi^*$ according to (\ref{eq:sequence})   we may define
a policy  $\pi^{*}=(\phi^{*}_0,\ldots,\phi^{*}_{N-1})$ for the original mean-variance model and the measures 
$\mu_1,\ldots,\mu_N.$   More precisely,
$\phi^{*}_n=\hat\phi^*_n(\mu_n)$ where $\mu_n$ is the distribution of the $n$-th state and  
\begin{equation}
\label{Eq:optphi}
\phi^{*}_{n}(x)= 
\left(\frac{\lambda  S^0_n}{d_n S^0_N }+\int x'\mu_n(dx')  -x\right)C_{n+1}^{-1}\E R_{n+1}, \quad x\in E,
\end{equation}
$n=0,\ldots, N-1.$ 

\begin{corollary} If $\mu_0=\delta_{x_0},$ then the value 
 $J_0(\delta_{x_0})$
is the value of $P(\lambda)$ and  $J_0(\delta_{x_0})=V^{o}(x_0).$ 
The  policy $\pi^{*}=(\phi^{*}_0,\ldots,\phi^{*}_{N-1})$ with $\phi^{*}_n$ 
defined in  \eqref{Eq:optphi} for $n=0,\ldots, N-1$ is optimal.
These decision rules are history-dependent, however depend at time $n$ on the history only through $\mu_n.$ 
\end{corollary}

Furthermore, we observe that the expression $\int x' \mu_n(dx')$  in \eqref{Eq:optphi}
is exactly $\E_{\mu_0}^{\pi^{*}}[X_n].$ Thus, we can compute it as follows:
\begin{align*}
    \E_{\mu_0}^{\pi^{*}}[X_1] &= (1+i_1) \E_{\mu_{0}} \Big[ X_{0}+\phi_0^{*}(X_{0}) \cdot R_1\Big]
     = (1+i_1) \Big( \E_{\mu_{0}} [ X_{0}] +\frac{\lambda}{d_0 S_N^0} \ell_1\Big)\\
    &= S_1^0\Big( \E_{\mu_0} [X_0] + \frac{\lambda}{S_N^0} \frac{\ell_1}{d_0}\Big).
\end{align*}
For arbitrary time $n=1,\ldots, N,$ we get
$$  \E_{\mu_0}^{\pi^{*}}[X_n] = S_n^0 \Big(\E_{\mu_0} [X_0] 
+ \frac{\lambda}{S_N^0} \sum_{k=1}^n \frac{\ell_k}{d_{k-1}}\Big).$$
Next we  insert this in  \eqref{Eq:optphi} to obtain for $n=0,\ldots,N-1$:
\begin{eqnarray}\nonumber
\phi^{*}_n(x) &=&  \left(S_n^0\E_{\mu_0} [X_0] 
+ \lambda  \frac{  S^0_n}{S^0_N} \Big( \sum_{k=1}^n \frac{\ell_k}{d_{k-1}}
+\frac{1}{d_n}\Big) -x\right)C_{n+1}^{-1}\E R_{n+1}\\ \label{Eq:optphi++}
&=&  \left( \Big(\E_{\mu_0} [X_0] +   \frac{\lambda}{S^0_N d_0}\Big) S_n^0  -x\right)C_{n+1}^{-1}\E R_{n+1},\quad x\in E,
\end{eqnarray}
since for all $n=1,\ldots,N-1$: 
$$\sum_{k=1}^n \frac{\ell_k}{d_{k-1}}+\frac{1}{d_n}
=\sum_{k=1}^n \Big(\frac{1}{d_{k-1}}-\frac{1}{d_k} \Big)+\frac{1}{d_n}= \frac{1}{d_0}.$$
Thus,  if we consider $\mu_0=\delta_0,$ then the optimal policy in \eqref{Eq:optphi++} coincides 
with the optimal policy $\pi^{o}$ in Proposition \ref{sol1}. Moreover, 
 the optimal values are the same, i.e., 
the use of further information does not improve the value. Comparing the two approaches we may point out 
 two advantages for the latter one. First, 
we were  able to solve  our problem $P(\lambda)$ directly without using auxiliary problems.
We need only forward  recursion  formulae in Definition \ref{def1}. Second, in the {\it population version} model, 
this policy can be viewed as {\it time-consistent,} i.e., 
it is optimal and satisfies the Bellman optimality principle. 

Compared to the second solution 
with the equilibrium strategy $\pi^{e}$, it is already clear by 
construction  that the value obtained in the {\it population 
version} MDP cannot be worse. This is because the policy in  \eqref{Eq:optphi}
can be seen as a  history-dependent policy  in the original formulation.
This policy  depends on the process history only via the probability 
distribution, i.e., $\mu_n$ is a function of 
$(X_0,A_0,\ldots,A_{n-1}).$ The equilibrium strategy, on the 
other hand, is completely deterministic.




\section{General non-additive Markov decision processes }\label{sec:general}
In this section, we consider a generalization of the mean-variance problem.  
We show that  a general time-inconsistent  problem 
can be solved yielding a time-consistent policy in the {\em population version } MDP. 
Suppose we are given an original  MDP with finite time horizon $N$ and with the following data:
\begin{itemize}
\item $(E,\mathcal{B}(E))$  is a Borel  state space; the states are denoted by $x\in E,$
\item $(A,\mathcal{B}(A))$ is a Borel action space; the actions are denoted by $a\in A$,
\item $(\mathcal{R},\mathcal{B}(\mathcal{R})) $ is some Borel space; the random 
variables $R_1,\ldots ,R_N$ are independent with values in $\mathcal{R}$ defined on the joint probability space
 $(\Omega,\mathcal{F},\PP)$,
\item $T_n(x,a,r)$ is a measurable transition function at time $n=0,\ldots ,N-1$, 
i.e., $T_n : E\times A\times \mathcal{R}\to E$ and $$X_{n+1} = T_n(X_n,A_n, R_{n+1})$$
is the new state at time $n+1,$ given the state at time $n$ is $X_n$ and action $A_n$ is taken; then
$$ Q_n(B|x,a)= \PP\big(T_n(x,a,R_{n+1})\in B\big), \quad B \in\mathcal{B}(E)$$
is the transition kernel,
\item $c_n(x,a)$ is a measurable cost 
function  $c_n:E\times A\to\R,$ which gives  
the cost at time $n=0,\ldots,N-1$ in state $x$ and action $a$,
\item $c_N(x)$ is the terminal cost in state $x$, i.e.,\ $c_N: E\to \R$,
\item $h:E\to\R $ and $G:\R\to\R$  are measurable functions providing the non-additivity. We set $G(\infty):=\infty.$
\end{itemize}
In what follows, let $\Pi$ be the set of all deterministic history-dependent policies $\pi=(\pi_0,\ldots,\pi_{N-1}),$  i.e., if 
$ h_n=(x_0,a_0,\ldots ,x_n)$ is the history up to time $n$, then $\phi_n(h_n)$ gives the action at time $n$. By 
 $$ F := \{ \phi: E \to A \ | \ \phi \mbox{ is measurable}\} $$ we denote the set of Markovian decision rules. 
 We are now interested in minimizing the following $N$-stage cost over all history-dependent policies:
 \begin{equation}\label{form:general}
 \mathbb{E}_{x_0}^\pi\Big[ \sum_{k=0}^{N-1} c_k(X_k,A_k)+ c_N(X_N)\Big]+
G(\mathbb{E}_{x_0}^\pi[h(X_N)])\to \inf_{\pi \in\Pi} \end{equation}
If $G$ is linear then this problem is a 
standard MDP and we can solve it with 
established techniques (see Remark \ref{rem:pe}). 
In what follows, we make no specific 
assumption about $G.$ In order to have a 
well-defined problem, we assume that a lower bounding function exists, i.e., 
there exists a measurable mapping $b:E\to \R$ and constants $\bar c_n,\bar c_h,\bar c_N,\alpha_b>0$ such that
\begin{itemize}
    \item[(i)] $c_n^-(x,a) \le \bar c_n b(x)$ for all $(x,a)\in E\times A, n=0,\ldots ,N-1$,
    \item[(ii)] $c_N^-(x)\le \bar c_N b(x)$ for all $x\in E$,
    \item[(iii)] $h^-(x)\le \bar c_h b(x)$ for all $x\in E$,
    \item[(iv)] $\int b(x') Q^\phi(dx'|x)\le \alpha_b b(x)$ for all $x\in E$ and decision rules $\phi$,
\end{itemize}
where $a^- = -\min\{0,a\},$ $a\in\R.$ 
This implies that all expectations are well-defined. 
We solve the problem again as a {\it population version} MDP: 
\begin{itemize}
\item $\hat E := \PP(E)$ is the state space, where $\mu\in \hat E$  is a probability measure on $E$,
\item $ \hat A :=F$ is the action space, where $\phi\in \hat A$ is a decision rule for the original MDP,
\item $\hat T_n(\mu, \phi)=\mu'$ is the transition function, where 
\begin{align*}
    \mu'(B) &= \int \PP\big(T_n(x,\phi(x),R_{n+1}) \in B\big)\mu(dx) \\ &
= \int Q_n(B|x,\phi(x)) \mu(dx)= \int Q_n^\phi(B|x) \mu(dx)
\end{align*}  
and $Q_n^\phi(B|x)=Q_n(B|x,\phi(x)),$ for  $B\in\mathcal{B}(E)$ and $x\in E$,
\item $\hat c_n(\mu,\phi)= \int c_n(x,\phi(x))\mu(dx) $ is the one-stage cost,
\item $\hat c_N(\mu) = \int c_N(x) \mu(dx) + G(\int h(x)\mu(dx))$ is the terminal cost.
\end{itemize}
This MDP is now a deterministic dynamic program. 
We know that we can restrict the optimization to Markovian policies $\hat\pi=(\hat\phi_0,\ldots ,\hat\phi_{N-1}),$ where
each $\hat\phi_n\in\hat F$ for $n=0,\ldots,N-1,$ and 
$\hat F=\{\hat\phi:\hat E\to\hat A\}.$  We again do not require any measurability conditions of the mappings in $\hat F.$

\begin{definition} \label{def2}
Given a policy
 $\hat \pi = (\hat\phi_0,\ldots,\hat\phi_{N-1})$ we obtain the following state-action sequence:
\begin{align}
\nonumber
 \phi_0 &:= \hat\phi_0(\mu_0),\quad 
 \mu_1 := \hat T_0(\mu_0,\phi_0),\\  \nonumber
   \phi_1 &:= \hat\phi_1(\mu_1),\quad
 \mu_2 := \hat T_1(\mu_1,\phi_1),\\  \nonumber
 &\vdots\\ \label{eq:sequence2}
 \phi_{N-1} &:= \hat\phi_{N-1}(\mu_{N-1}),\quad \mu_{N} := \hat T_{N-1}(\mu_{N-1},\phi_{N-1}).
\end{align} 
The sequence of actions defines a policy 
$\pi=(\phi_0,\phi_1,\ldots,\phi_{N-1})$ in 
the original MDP.
\end{definition}

\begin{lemma} \label{lemp3}
Under policy $\hat\pi=(\hat\phi_0,\ldots ,\hat\phi_{N-1})$, the distribution of $X_k$ is equal to $\mu_k$,
where $\mu_k$ are defined in (\ref{eq:sequence2}) for $k=0,1,\ldots, N.$
\end{lemma}

\begin{proof}
 The proof goes by induction as in Lemma \ref{lemp1} with the obvious changes.
\end{proof}

The value functions for a fixed policy $\hat\pi=(\hat\phi_0,\ldots ,\hat\phi_{N-1})$ are given by
\begin{eqnarray*}
 J_{N,\hat\pi}(\mu) &:=& \hat c_N (\mu),\\
 \mbox{and}\quad J_{n,\hat\pi}(\mu) &:=& 
 \hat c_{n}(\mu,\hat\phi_{n}(\mu)) + J_{n+1,\hat\pi}(\hat T_{n}(\mu,\hat\phi_{n}(\mu))),\quad n=0,\ldots,N-1.
\end{eqnarray*}

\begin{lemma} \label{lemp4}
Under policy $\hat \pi = (\hat\phi_0,\ldots,\hat\phi_{N-1})$, 
$$J_{0,\hat\pi}(\mu_0) =
 \mathbb{E}_{\mu_0}^\pi\Big[ \sum_{k=0}^{N-1} c_k(X_k,A_k)+ c_N(X_N)\Big]
 +G(\mathbb{E}_{\mu_0}^\pi[h(X_N)]),$$
where $\pi$ on the right-hand side is defined as in \eqref{eq:sequence2}.
\end{lemma}

\begin{proof}
We obtain
\begin{eqnarray*}
J_{0,\hat\pi}(\mu_0) &=& \hat c_0(\mu_0,\hat\phi_0(\mu_0)) + 
J_{1,\hat\pi}(\hat T_0(\mu_0,\hat\phi_{0}(\mu_0)))= \hat c_0(\mu_0,\phi_0) + J_{1,\hat\pi}(\hat T_0(\mu_0,\phi_{0}))\\
&=& \EE_{\mu_0}[c_0(X_0,\phi_0(X_0))]+  J_{1,\hat\pi}(\mu_1)\\
&=&  \EE_{\mu_0}[c_0(X_0,\phi_0(X_0))]+  \EE_{1,\mu_1}[c_1(X_1,\phi_1(X_1))] + J_{2,\hat\pi}(\hat T_1(\mu_1,\phi_{1}))\\
  &=&  \EE_{\mu_0}[c_0(X_0,\phi_0(X_0))]+ \ldots +  \EE_{N-1,\mu_{N-1}}[c_{N-1}(X_{N-1},\phi_{N-1}(X_{N-1}))]+ \hat c_N(\mu_N)\\
  &=&  \mathbb{E}_{\mu_0}^\pi\Big[ \sum_{k=0}^{N-1} c_k(X_k,\phi_k(X_k))+ c_N(X_N)\Big]+G(\mathbb{E}_{\mu_0}^\pi[h(X_N)]),
\end{eqnarray*}
where we have used   Lemma \ref{lemp3} and \eqref{eq:sequence2}.
\end{proof}

Finally, we define for $n=1,\ldots,N$:
\begin{align}\nonumber
 J_{N}(\mu) &:= \hat c_N(\mu),\\ \label{eq:recursion2}
 J_{n}(\mu) &:= \inf_{\hat\phi\in \hat F}\left( \hat c_{n}(\mu,\hat\phi(\mu)) + J_{n+1}(\hat T_{n}(\mu,\hat\phi(\mu)))\right).
\end{align}
If the  minimum in \eqref{eq:recursion2} is
attained, we denote it by $ \hat\phi_{n}^*(\mu).$ Note that the recursion in  \eqref{eq:recursion2} 
is exactly the standard Bellman equation, 
because on the right-hand we have the sum of the current cost and the value function,
whereas on the left hand side we have the same value function. This means that the optimal policy in the {\it population version}
 MDP (if it exists) is {\it  time-consistent} according to Definition \ref{tc}.
Thus, we obtain as in the previous section:

\begin{theorem}  \label{thm:m} Assume that the initial distribution is $\mu_0=\delta_{x_0}.$
The value $J_0(\mu_0)$ 
obtained with the recursion in \eqref{eq:recursion2} is 
the value of problem \eqref{form:general} and the optimal policy (if it exists) 
 $\hat \pi^*=(\hat\phi^*_0,\ldots ,\hat\phi^*_{N-1})$ is time-consistent in the population version MDP.
The policy $\pi^*=(\phi^*_0,\ldots,\phi^*_{N-1})$  defined in \eqref{eq:sequence2}, i.e., 
$\phi^*_k=\hat\phi^*_k(\mu_k)$ for $k=0,\ldots,N-1$ is optimal in the original formulation. 
These decision rules are history-dependent, however depend at time $n$ 
on the history only through $\mu_n$.
\end{theorem}

Theorem  \ref{thm:m}  means that the optimal policy in the original formulation can be computed by 
a backward-forward algorithm. First, we use \eqref{eq:recursion2} to determine the value functions $J_n$ 
and the optimal policy $\hat \pi^*.$ Then, we  use \eqref{eq:sequence2} to find 
the sequence of measures $\mu_0,\ldots,\mu_{N-1}$ and  the policy $\pi^*$.

\begin{remark}\label{rem:pe}
In case $G(x)=x$ (w.l.o.g.\ we can then set $c_N\equiv 0$) and in case 
minimizers exist, our solution method boils down to the classical Bellman equations of MDP theory. 
Indeed, in this case it can be shown by induction that $J_n(\mu)= \int V_n(x) \mu(dx), n=0,1,\ldots,N$
 for some measurable functions $V_n.$ The proof is as follows. First, for $N$ we have by definition 
 $J_N(\mu)= \int h(x)\mu(dx)$. Suppose $J_{n+1}(\mu)= \int V_{n+1}(x) \mu(dx).$ Then
\begin{align*}
J_n(\mu) &= \inf_{\phi\in F} \left(\int c_n(x,\phi(x)) \mu(dx)+ \int \EE \big[V_{n+1}(T(x,\phi(x),R_{n+1}))\big] \mu(dx)\right)\\
&=  \inf_{\phi\in F} \left(\int c_n(x,\phi(x)) +  \EE \big[V_{n+1}(T(x,\phi(x),R_{n+1}))\big] \mu(dx)\right)\\
&= \int c_n(x,\phi^*(x)) +  \EE \big[ V_{n+1}(T(x,\phi^*(x),R_{n+1})) \big]\mu(dx)
\end{align*} 
which implies the statement. Note that, in particular,  the 
minimization can be done pointwise, which yields the classical
Bellman equation
$$ V_n(x) = \inf_{\phi\in F} \Big(c_n(x,\phi(x)) +  \EE \big[ V_{n+1}(T(x,\phi(x),R_{n+1}))\big]\Big).$$
\end{remark}

\section{Application: LQ-problem}\label{sec:application}
In this section,  we apply our findings to the following LQ-problem considered in  \cite{bjork2021}, Sec. 9:
\begin{equation}
 \mathbb{E}_{\mu_0}^\pi\Big[ \sum_{k=0}^{N-1} A_k^2 \Big]
 +(\mathbb{E}_{\mu_0}^\pi[X_N])^2\to\inf_{\pi} \label{form:lq}
 \end{equation}
 where the next state evolves according to the equation $$X_{n+1}=b X_n+d A_n +\sigma R_{n+1}.$$ 
 Here $X_n\in E=\R,$ 
 $A_n\in A=\R$ 
 and $R_1,\ldots ,R_N$ 
 are i.i.d. random variables with a finite 
 second moment and $\E R_n=0$ for $n=1,\ldots,N$, 
 $b,d \in \R,$ $\sigma>0.$
Hence, within this framework  in the 
{\it population version} MDP, the transition function is
$\hat T_n(\mu, \phi)=\mu',$ where 
$$\mu'(B)= \int \PP(T_n(x,\phi(x),R_{n+1}) \in B)\mu(dx) 
= \int \PP(b x+ d \phi(x)+\sigma R_{n+1} \in B)\mu(dx)  $$ for a Borel set $B\subset E.$ 
Moreover, $$\hat c_n(\mu,\phi)= \int \phi(x)^2 \mu(dx) $$ is the one-stage cost, whereas
 $$\hat c_N(\mu) =  \Big(\int x\mu(dx)\Big)^2$$ is the terminal cost.

In order to solve the problem, we first 
consider the final time point, where 
$$J_N(\nu) = \Big(\int x\nu(dx)\Big)^2= \Big(\E_{N,\nu}[X_N]\Big)^2.$$ 
At time point $N-1,$  equation \eqref{eq:recursion2} becomes
\begin{align*}
J_{N-1}(\nu) &= \inf_{\phi\in F}\left( \E_{N-1,\nu}[\phi^2(X_{N-1})] + J_{N}(\hat T_{N-1}(\nu,\phi))\right)\\
 &=  \inf_{\phi\in F}\left( \E_{N-1,\nu}[\phi^2(X_{N-1})] + \Big(b \E_{N-1,\nu}[X_{N-1}]
 +d \E_{N-1,\nu}[\phi(X_{N-1})]  \Big)^2\right).
\end{align*} 
The minimizer is given by (see Appendix \ref{sec:appendix2})
$$ \hat\phi^*_{N-1}(\nu)(x) = - \frac{bd}{1+d^2}  \E_{N-1,\nu}[X_{N-1}],\quad \nu\in\hat E,\ x\in E.$$
Plugging this minimizer into the equation for $J_{N-1}$ we obtain:
$$  J_{N-1}(\nu) = \frac{b^2}{1+d^2}  \Big(\E_{N-1,\nu}[X_{N-1}]  \Big)^2.$$
Thus, we get again a quadratic form. Using  
the results in Appendix \ref{sec:appendix2}, it can be shown by induction 
that the optimal policy and value functions are given for $\nu\in\hat E$ by
\begin{align}\label{op:lq}
   \hat \phi^*_n(\nu)(x) &= -\frac{bd \beta_{n+1}}{1+d^2 \beta_{n+1}} \E_{n,\nu}[X_n], \quad x\in E,\\
    J_n(\nu) &= \beta_n \big( \E_{n,\nu}[X_n]\big)^2, \nonumber
\end{align}
where the constants $\beta_n$ are determined by the backward recursion 
$$\beta_N=1, \quad \quad
\beta_ n = \frac{\beta_{n+1} b^2}{1+d^2 \beta_{n+1}}.$$    
Hence, making use of Theorem \ref{thm:m} we conclude:

\begin{proposition} \label{plq:m} If
the initial distribution $\mu_0=\delta_{x_0},$ then 
 $J_0(\mu_0)=\beta_0  x_0^2$ 
 is the value of problem \eqref{form:lq} and the optimal policy  
 $\hat \pi^*=(\hat\phi^*_0,\ldots ,\hat\phi^*_{N-1})$ defined in (\ref{op:lq}) 
 is time-consistent in the population version MDP.
The policy $\pi^*=(\phi^*_0,\ldots,\phi^*_{N-1})$  defined as in \eqref{op:lq}, i.e., 
$$\phi^*_n(x)=\hat\phi^*_n(\mu_n)(x)\equiv -
\frac{bd \beta_{n+1}}{1+d^2 \beta_{n+1}} \E_{\mu_n}[X_n],\quad x\in E,$$
for $n=0,\ldots,N-1$,
is optimal in the original formulation. 
\end{proposition}

Since the sequence of states (measures) evolves in a deterministic way,
we can compute $ \E_{\mu_n}[X_n]$ by a forward recursion as follows 
\begin{align*}
   \E_{\mu_1}[X_1] &=  \E_{\mu_0}[b X_0 + d \phi_0(X_0)] = b  \E_{\mu_0}[X_0] -
 \frac{bd^2 \beta_1}{1+d^2\beta_1}  \E_{\mu_0}[X_0]  = \E_{\mu_0}[X_0]  \frac{b}{1+d^2\beta_1}
\end{align*}
and in general
$$  \E_{\mu_n}[X_n] =  \E_{\mu_0}[X_0] \frac{b^n}{\prod_{k=1}^n (1+d^2\beta_k)}.$$
Thus, we get that (in case $\mu_0=\delta_{x_0}$)
\begin{align*}  
\phi^*_n(x) &= -\frac{b^{n+1}d \beta_{n+1}} {\prod_{k=1}^{n+1} (1+d^2\beta_k)} \E_{\mu_0}[X_0]
=-\frac{b^{n+1}d \beta_{n+1}}{\prod_{k=1}^{n+1} (1+d^2\beta_k)} x_0,\quad x\in E,\\
 J_n(\mu_n) &= \beta_n \Big( \E_{\mu_0}[X_0] \frac{b^n}{\prod_{k=1}^n (1+d^2\beta_k)} \Big)^2=
 \beta_n \Big( x_0 \frac{b^n}{\prod_{k=1}^n (1+d^2\beta_k)} \Big)^2.
\end{align*}
Hence, the optimal policy in the original model is semi-Markov. Although it is independent of the current state, it does depend on the initial state.

We now compare our solution to the equilibrium strategy given 
in \cite{bjork2021}, 
where the value function is computed 
by backward induction with the help of a system   \`a la Bellman equations. For this approach we also have to assume that the variance of $R_n$ exist and is equal to 1. The value functions are of the form 
$$ V^e_n(x)= \beta_n x^2 + \gamma_n,\quad x\in E,$$
for $n=0,\ldots,N.$\footnote{There is a typo in the recursion for $A_n$ on p. 97 in \cite{bjork2021}. In our setting, $A_n$ is exactly $\beta_n$.} 
The constants $\gamma_n$ are 
computed recursively from $\gamma_N=0$ by
$$    \gamma_n = \gamma_{n+1}+\sigma^2(\beta_{n+1}-\alpha_{n+1}^2)$$
with  $\alpha_N=1$ and
\begin{align*}
    \alpha_n &= \alpha_{n+1}\frac{b}{1+d^2\beta_{n+1}}.
\end{align*}
The equilibrium strategy is given by
\begin{equation}
\label{equ_pol_LQ}
\varphi_n(x) = -\frac{bd\beta_{n+1}}{1+d^2 \beta_{n+1}}x, \quad x\in E,
\end{equation}
for $n=0,\ldots,N-1.$

In order to compare the two approaches, consider the value functions under our optimal policy and the equilibrium strategy in \eqref{equ_pol_LQ}, when the initial state is $x_0.$ Then, $J_0(\delta_{x_0})=\beta_0 x_0^2$ and $V^e_0(x_0)=\beta_0x_0^2+\gamma_0$.
However, the  function $V^e_0(x_0)$ depends on $\sigma$ through the constant $\gamma_0$ 
and therefore, 
it can be arbitrary large under the equilibrium policy in  (\ref{equ_pol_LQ}).
Note that $\gamma_0>0,$ which can be proved
by induction. Indeed, we show that $\beta_n \ge \alpha_n^2,$ where the inequality is strict for $n=0,\ldots,N-1$. Obviously $\beta_N=\alpha_N^2.$ Assume that $\beta_{n+1} > \alpha_{n+1}^2,$ Thus, for $b,d\neq 0,$
we have 
\begin{align*}
    \beta_n -\alpha_n^2 &=  \frac{b^2 \beta_{n+1}}{1+d^2\beta_{n+1}} -  \alpha_{n+1}^2\frac{b^2}{(1+d^2\beta_{n+1})^2} \\
    &= \frac{b^2}{(1+d^2\beta_{n+1})^2} \Big( \beta_{n+1}(1+d^2\beta_{n+1}) -\alpha_{n+1}^2\Big) >0.
\end{align*}
The statement easily follows, because  $\beta_{n+1}> 0.$ Hence, $J_0(\delta_{x_0})<V^e_0(x_0),$ which implies that
the equilibrium strategy obtained by backward induction 
in \cite{bjork2021} is suboptimal. 
For instance, taking $x_0=0$
the best thing one can do is take $A_n\equiv 0,$ for $n=0,\ldots,N-1.$ In this case, $ \phi^*_n(x)\equiv 0$ for every $x\in E$ and $n=0,\ldots, N-1,$ and  $J_0(\delta_{x_0})=0.$ 
However, under the policy in (\ref{equ_pol_LQ}), we have $V^e_0(0)=\gamma_0>0.$

\begin{remark}
When we make some distributional assumptions about the random variables 
$R_n$, our approach 
simplifies in some cases. For example, suppose that in addition $R_n\sim 
N(0,\tau_n^2)$ is normally distributed, $X_0\sim N(\rho_0,\sigma_0^2)$ is 
normally distributed and we restrict to linear decision rules 
$\phi_n(x)=m_n x+t_n$ with some $m_n,t_n\in\R.$ Then it is easy to see by 
induction that $X_n\sim N(\rho_n,\sigma_n^2)$ is again normally 
distributed. The parameters evolve according to
    \begin{align*}
        \rho_{n+1} &= (b+dm_n)\rho_n+d t_n,\\
        \sigma_{n+1}^2 &= (b+dm_n)^2  \sigma_n^2+\sigma^2 \tau_{n+1}. 
    \end{align*}
In this case, it is then enough to choose as a state the pair: mean and variance, 
since it determines the distribution already. Thus, we can work  with a 
state space of lower dimension.
\end{remark}


\section{Extensions and Connection to mean-field MDPs}\label{sec:meanfield}
The problems considered in Section \ref{sec:general} can be generalized in various ways. First, instead of deterministic, Markovian decision rules in the original MDP one could also consider randomized Markovian decision rules. They are given by transition kernels from $E$ to $A,$ i.e. $x\mapsto \phi(x)(B)$ is measurable for every measurable $B\subset A$ and $\phi(x)$ is a probability distribution on $A.$ This translates in the population version to a transition law $\hat T_n(\mu,\phi)=\mu'$ given by
$$\mu'(B)= \int\int \PP(T_n(x,a,R_{n+1}) \in B)\mu\otimes\phi(d(x,a)) 
 $$ for a Borel set $B\subset E$ and
 $$\hat c(\mu,\phi)= \int\int c(x,a)\mu\otimes\phi(d(x,a)).  $$
The interpretation behind this is that at time $n,$ when individuals are distributed according to $\mu_n,$ all individuals in state $x$ flip a coin according to $\phi(x)$ to determine their actions.

Another generalization would be to allow the transition function and the cost function to depend in a more sophisticated way on $\mu.$ For example $\hat c(\mu,\phi)$ could be a quantile of $\mu$ or an entropy measure.

Also we have not considered constraints on actions. Often in MDP theory the admissible actions are constrained given the state. It would be possible to model probability constraints in such a way that the marginal probability that the process stays above a threshold $\beta$  is bounded  below by $\alpha\in(0,1)$. More precisely we could define
$$D(\mu) := \Big\{ \phi\in F : \hat \mu := \hat T(\mu,\phi) \mbox{ satisfies } \int 1_{[x\ge \beta]}\hat\mu(dx) \ge \alpha\Big\} $$
as the set of admissible actions such that in the next step the probability constraint is satisfied.

All these generalizations lead to so-called mean-field MDPs, where the behavior and interaction of an infinite number of cooperative players is modeled (see e.g. \cite{bauerle2021mean,pham2016discrete}). Of course other target criteria may also be considered.

\section{Conclusion}\label{sec:conclusion}
We suggested a new approach to non-separable MDPs which we called {\em population version} and which is inspired by mean-field MDPs. In this approach we consider the distribution process of the states instead of the states themselves. The advantage of this method is twofold. First it admits a Bellman optimality equation for the solution of the problem and second it provides a time-consistent policy. We demonstrated its usefulness with examples taken from the realm of LQ-problems. 

\section{Appendix:  Solution of the one step  problem}\label{sec:appendix}
\subsection{Mean-variance problem }\label{sec:appendix1}
In this section, we find a solution to (\ref{eq:onestep}). 
Our objective is to find an element $\phi\in F$ (if exists) that minimizes
$$(1+i)^2\left[
\E_{\nu} \Big[(X +\phi(X) \cdot R)^2 -\frac{2 \lambda} {1+i}(X+\phi(X)\cdot R) \Big]
-\Big(\E_{\nu} \Big[X+\phi(X)\cdot R\Big]\Big)^2\right].$$
Here, we skip the time period. Recall that $X\sim \nu.$ Since $\E [Y] =\mbox{arg}\min_{b\in\R}\E(Y-b)^2,$
 we may write our problem as follows
\begin{equation}\label{optim:onestep}
(1+i)^2 \inf_{\phi\in F}\min_{b\in\R}
\E_{\nu} \Big[(X +\phi(X) \cdot R-b)^2 -\frac{2 \lambda} {1+i}(X+\phi(X)\cdot R) \Big].\end{equation}
First we solve the outer problem and then the inner one, 
because we may interchange the infima. Assume that $X=x$ and $y=\phi(x).$
Then, we have 
$$
\inf_{y\in\R^d}
\E \Big[(x +y \cdot R-b)^2 -\frac{2 \lambda} {1+i}(x+y\cdot R) \Big].
$$
The solution of this problem is
$$y=\phi(x)=\Big(b+\frac\lambda{1+i} -x\Big)C^{-1}\E R.$$ 
where $C=\E[RR^\top].$ 
Hence, plugging $\phi(X)=\Big(b+\frac\lambda{1+i} -X\Big)C^{-1}\E R$ into   (\ref{optim:onestep}) 
we have to find the minimizer  with respect to $b\in\R$ for
\begin{eqnarray*}\lefteqn{
\E_{\nu} \Big[X^2-\frac{2\lambda}{1+i}X
+b^2-2Xb+}\\&&
\left(b+\frac\lambda{1+i}-X\right)^2(R^\top C^{-1}\E R)^2+\left(2X-2b-\frac{2\lambda}{1+i}\right)
\left(b+\frac\lambda{1+i}-X\right)R^\top C^{-1}\E R\Big]\\&=&
\E_{\nu} \Big[X^2-\frac{2\lambda}{1+i}X
-2bX+\left(b+\frac\lambda{1+i}-X\right)^2
((R^\top C^{-1}\E R)^2-2R^\top C^{-1}\E R)\Big]+b^2\\&=&\E_{\nu} \Big[X^2-\frac{2\lambda}{1+i}X\Big]
-2b\E_\nu X-\E_\nu
\left(b+\frac\lambda{1+i}-X\right)^2
(\E R)^\top C^{-1}\E R +b^2,
\end{eqnarray*}
where  the last equality is due to the independence of $R$ and $X.$
The solution is 
$$b=\E_\nu X+\frac\lambda{1+i}\ \frac{(\E R)^\top C^{-1}\E R }{1-(\E R)^\top C^{-1}\E R}.$$
This implies that
$$\phi(x)=\Big(\E_\nu X+\frac\lambda{1+i}\ \frac{1 }{1-(\E R)^\top C^{-1}\E R}-x\Big)C^{-1}\E R.$$
 Hence, the solution for the the last step is
\begin{eqnarray*}
\phi_{N-1}(\nu)(x)&=&\Big(\E_{N-1,\nu} X+\frac\lambda{1+i_N}\ \frac{1 }{1-
(\E R_N)^\top C_N^{-1}\E R_N}-x\Big)C_N^{-1}\E R_N\\&=&
 \Big(\E_{N-1,\nu} X+\frac\lambda{(1+i_N)(1-\ell_N)}-x\Big)C_N^{-1}\E R_N,
\end{eqnarray*}
where we  use (\ref{eq:ell}).

\subsection{LQ-problem}\label{sec:appendix2}
For the following computation we skip the time index $N-1$ and suppose that $\beta>0$ is an additional parameter. We consider
\begin{align*}
    J(\nu) &=  \inf_{\phi\in F} \left(\E_{\nu}[\phi^2(X)] + \beta \Big(b \E_{\nu}[X]+d \E_{\nu}[\phi(X)]  \Big)^2\right) \\
    &= \beta b^2 \big(\E_{\nu}[X]\big)^2+ \inf_{\phi\in F}  \left(\E_{\nu}[\phi^2(X)] + 2bd \beta \E_{\nu}[X]   \E_{\nu}[\phi(X)] + \beta d^2 \Big( \E_{\nu}[\phi(X)]  \Big)^2\right).
\end{align*} 
For the moment suppose that  a minimizer exists and has a certain expectation $\E_{\nu}[\phi(X)]=\bar \phi.$ We have to choose $\phi$ such that the second moment $\E_{\nu}[\phi^2(X)] $ is minimized. Since by the Jensen inequality
$$ \E_{\nu}[\phi^2(X)] \ge \Big(  \E_{\nu}[\phi(X)]\Big)^2 $$
it is optimal to choose $\phi(X)\equiv \bar \phi.$ Hence, we have to solve
\begin{align*}
    J(\nu) &=  \beta b^2 \big(\E_{\nu}[X]\big)^2+ \inf_{\bar\phi\in \R} \left(\bar\phi^2  + 2bd \beta \E_{\nu}[X]   \bar\phi + \beta d^2 \bar\phi^2\right).
\end{align*} 
The minimum point is given at
$$\hat \phi(\nu)(x)\equiv \bar\phi=- \frac{bd\beta}{1+d^2\beta}  \E_{\nu}[X].$$
If we plug this into the equation for $J$ we obtain
\begin{align*}
    J(\nu) &=  \beta b^2 \big(\E_{\nu}[X]\big)^2+ \frac{\big(bd\beta   \E_{\nu}[X]\big)^2}{(1+d^2\beta)^2} - 2  \frac{\big(bd\beta  \E_{\nu}[X]\big)^2}{1+d^2\beta} +\beta d^2 \frac{\big(bd \beta  \E_{\nu}[X]\big)^2}{(1+d^2\beta)^2}\\
    &=  \beta b^2\big(\E_{\nu}[X]\big)^2  \Big( 1+ \frac{d^2\beta}{(1+d^2\beta)^2}-2 \frac{d^2\beta}{1+d^2\beta}+ \frac{d^4\beta^2}{(1+d^2\beta)^2}\Big)\\
    &=  \big(\E_{\nu}[X]\big)^2 \frac{\beta b^2}{1+d^2\beta}.
\end{align*}

\bibliographystyle{amsplain}
\bibliography{literature}

\end{document}